\documentclass[11pt]{amsart}
\usepackage{geometry}
\usepackage{amsmath}
\usepackage{amsfonts}
\usepackage{amsthm}
\usepackage{amssymb}
\allowdisplaybreaks[2]

\newcommand{\D}{\mathbb D}
\newcommand{\T}{\mathbb T}
\newcommand{\N}{\mathbb N}

\newcommand{\C}{\mathbb C}
\newcommand{\cb}{\mathcal B}

\newcommand{\cd}{\mathcal D}

\newcommand{\Dp}{\mathcal D^p_{p-1}}
\newcommand{\Dq}{\mathcal D^q_{q-1}}
\newcommand{\Dpa}{\mathcal D^p_{\alpha}}

\newcommand{\hol}{\mathcal Hol}

\DeclareMathOperator{\og}{O}
\DeclareMathOperator{\op}{o}

\newcommand{\iii}{\frac1{2\pi}\int_0^{2\pi}}
\newcommand{\ig}{\stackrel{\text{def}}{=}}

\newcommand{\equstart}{\begin{equation}\begin{aligned}}
\newcommand{\equend}{\end{aligned}\end{equation}}
\newcommand{\equstartu}{\begin{equation*}\begin{aligned}}
\newcommand{\equendu}{\end{aligned}\end{equation*}}
\newtheorem{theorem}{Theorem}

\newtheorem{proposition}{Proposition}
\theoremstyle{definition}

\theoremstyle{remark}

\numberwithin{equation}{section}
\theoremstyle{plain}
\newtheorem{other}{Theorem}         
\newtheorem{otherc}[other]{Corollary}

\begin{document}
\title[Superposition operators, Hardy spaces, and Dirichlet type spaces]
{Superposition operators, Hardy spaces, and Dirichlet type spaces}

\author[P.~Galanopoulos]{Petros Galanopoulos}
\address{Department of Mathematics,
Aristotle University of Thessaloniki, 54124 Thessaloniki, Greece}
 \email{petrosgala@math.auth.gr}
\author[D.~Girela]{Daniel Girela}
\address{An\'alisis Matem\'atico,
Universidad de M\'alaga, Campus de Teatinos, 29071 M\'alaga, Spain}
\email{girela@uma.es}
\author[M.~A.~M{\'a}rquez]{Mar{\'\i}a Auxiliadora M{\'a}rquez}
\address{An\'alisis Matem\'atico,
Universidad de M\'alaga, Campus de Teatinos, 29071 M\'alaga, Spain}
\email{auxim@uma.es}

\date{}

\begin{abstract}
For $0<p<\infty $ and $\alpha >-1$ the space of Dirichlet type
$\mathcal D^p_\alpha $ consists of those functions $f$ which are
analytic in the unit disc $\mathbb D$ and satisfy $\int_{\mathbb
D}(1-\vert z\vert )^\alpha\vert f^\prime (z)\vert ^p\,dA(z)<\infty
$. The space $\Dp$ is the closest one to the Hardy space $H^p$ among
all the $\mathcal D^p_\alpha $. Our main object in this paper is
studying similarities and differences between the spaces $H^p$ and
$\Dp$ ($0<p<\infty $) regarding superposition operators. Namely, for
$0<p<\infty $ and $0<s<\infty $, we characterize the entire
functions $\varphi $ such that the superposition operator $S_\varphi
$ with symbol $\varphi $ maps the conformally invariant space $Q_s$
into the space $\Dp$, and, also, those which map $\Dp$ into $Q_s$
and we compare these results with the corresponding ones with $H^p$
in the place of $\Dp$.
\par We also study the more general question of characterizing the
superposition operators mapping $\mathcal D^p_\alpha $ into $Q_s$
and $Q_s$ into $\mathcal D^p_\alpha $, for any admissible triplet of
numbers $(p, \alpha , s)$.
\end{abstract}

\thanks{This research is supported by a grant from \lq\lq El Ministerio de Econom\'{\i}a y Competitividad\rq\rq , Spain (MTM2014-52865-P) and by a grant from la
Junta de Andaluc\'{\i}a FQM-210.}

\maketitle

\bigskip
\section{Introduction and main results}\label{intro}
Let $\D=\{z\in\C:|z|<1\}$ and $\T=\{z\in\C:|z|=1\}$
denote the open unit disc and the unit circle in the complex
plane $\C$, and let $\hol(\D)$ be the space of all analytic
functions in $\D$ endowed with the topology of uniform convergence
in compact subsets.
\par
If $0\le r<1$ and $f\in\hol (\D)$, we set
$$
M_p(r,f)=\left(\iii |f(re^{it})|^p\,dt\right)^{1/p},\,\text{if $
0<p<\infty$,\,\, and $M_\infty(r,f)=\sup_{|z|=r}|f(z)|.$}
$$
\par
For $0<p\le\infty$, the Hardy space $H^p$ consists of those
$f\in\hol(\D)$ such that
\[\|f\|_{H^p}\ig\sup_{0\le r<1}M_p(r,f)<\infty\]
(see \cite{D} for the theory of $H^p$-spaces).
If $0<p<\infty$ and $\alpha>-1$, the weighted
Bergman space $A^p_\alpha$ consists of those $f\in\hol(\D)$
such that
\[
\|f\|_{A^p_\alpha}\ig\left((\alpha+1)\int_\D(1-|z|^2)^\alpha|f(z)|^p\,dA(z)\right
)^{1/p}<\infty.\]
The unweighted Bergman space $A^p_0$ is
simply denoted by $A^p$. Here, $dA(z) =\frac{1}{\pi}dx\,dy$
denotes the normalized Lebesgue area measure in $\D$. We
refer to \cite{DS}, \cite{HKZ} and \cite{Zhu} for the theory of
these spaces.
\par
The space of Dirichlet type $\Dpa$ ($0<p<\infty$, $\alpha>-1$)
consists of those $f\in\hol(\D)$ such that $f'\in A^p_\alpha$.
Hence, if $f$ is analytic in $\D$, then $f\in\Dpa$ if and only if
\[\|f\|_{\Dpa}\ig|f(0)|+\|f'\|_{A^p_\alpha}<\infty.\]
If $\alpha>p-1$ then it is well known that $\Dpa=A^p_{\alpha-p}$
\cite[Theorem~6]{Flett}.
On the other hand, if $\alpha<p-2$ then $\Dpa\subset H^\infty$.
Therefore $\Dpa$ becomes a \lq\lq proper Dirichlet type space\rq\rq\
when $p-2\le\alpha\le p-1$.
For $p>1$, the space $\cd^p_{p-2}$ is the Besov space
$B^p$. The space $\cd^2_0$ is the classical Dirichlet
space $\cd$ of those
analytic functions $f$ in $\D$ satisfying $\, \int_\D|f'(z)|^2\,dA(z)<\infty$.
For these spaces, we have the following inclusions:
\begin{align}
&\cd^p_\alpha\subset\cd^q_\alpha,\quad\text{if $0<q<p<\infty$, $\alpha>-1$.}\label{DpalphaDqalpha}\\
&\cd^p_\alpha\subset\cd^p_\beta,\quad\text{if $0<p<\infty$, $-1<\alpha<\beta$.}\label{DpalphaDpbeta}
\end{align}
\par
The spaces $\Dp$ are the closest ones to Hardy spaces among all
the spaces $\Dpa$. It is well known that $\mathcal D^2_1=H^2$.
We have also \cite{LP}
\begin{equation}\label{1}
H\sp p\subsetneq\Dp,\quad\text{for $2<p<\infty$,}
\end{equation}
and \cite{Flett, Vi}
\begin{equation}\label{2}
\Dp\subsetneq H^p,\quad\text{for $0<p<2$.}
\end{equation}
Let us recall some similarities and differences between the spaces
$\Dp$ and Hardy spaces. Let us start with some similarities.
\begin{itemize}
\item
For $0<p\le 2$, the Carleson measures for $H^p$
and those for $\Dp$ are the same \cite{Vi, Wu}.
\item
For $0<p<\infty$, the univalent functions in $H^p$ and those in
$\Dp$ are the same \cite{BGP}.
\item
A number of operators are bounded on $H^p$ if and only if
they are bounded on $\Dp$, and, in some cases, it is useful to go
through $\Dp$ when one wants to study the action of an operator on
$H^p$ (see e.\,g.\ \cite{GGPS}).
\end{itemize}
Regarding differences, we may mention the following ones.
\begin{itemize}
\item
$H^\infty\not\subset\Dp$, if $0<p<2$. More precisely, there are Blaschke products not belonging
to any of the $\Dp$-spaces, $0<p<2$ \cite{Vi}.
\item
For $p>2$, there are functions in $\Dp$ without finite radial
limit almost everywhere \cite{GP:Aust06}.
\item
There is no inclusion relation between $\Dp$ and $\Dq$, $p\neq q$.
\item
The zero sequence of a $\Dp$-function, $p>2$, may not satisfy the Blaschke condition
\cite{GP:Aust06}.
\item
Even though it was conjectured in \cite{Wu} that the Carleson measures for $\Dp$ ($p>2$)
were the same, this is not true \cite{GP:IE06}.
\end{itemize}
\par
We refer to the recent paper \cite{PPR} for more results of this kind.
\par
\bigskip
Our first aim in this paper is studying similarities and differences between Hardy spaces and
$\Dp$-spaces regarding superposition operators.

Given an entire function $\varphi$, the superposition operator
\[S_\varphi:\hol(\D)\longrightarrow\hol(\D)\]
is defined by
$S_\varphi(f)=\varphi\circ f$.
We consider the following question:
If $X$ and $Y$ are two subspaces of $\hol(\D)$,
for which entire functions $\varphi$ does the operator $S_\varphi$
map $X$ into $Y$?

If $Y$ contains the constant functions, then it is clear that,
for constant functions $\varphi$, the operator $S_\varphi$ maps $X$ into $Y$.
It is also clear that for the function $\varphi(z)=z$, we have that
$S_\varphi(f)=f$ for all $f\in\hol(\D)$ and,
hence,
$S_\varphi(X)\subset Y$ if and only if $X\subset Y$.

If $X\subset Y$, then the answer to this question tells us \lq\lq how small\rq\rq\ is
the space $X$ inside the space $Y$.
Thus, if there are a lot of $\varphi$'s for which
$S_\varphi(X)\subset Y$ then we can say that $Y$ is
\lq\lq big\rq\rq\ compared with $X$.

The characterization of the entire functions $\varphi$ for which the superposition
operator maps $X$ into $Y$
has been studied for distinct pairs $(X,Y)$ of spaces of analytic functions.
For example,
\begin{itemize}
\item
C\'amera: Hardy spaces $H^p$ and $H^q$ (\cite{C}). \item C\'{a}mera
and Gim\'{e}nez: Bergman spaces $A^p$ and $A^q$ (\cite{CG}).
\item
Buckley, Fern\'andez and Vukoti\'c: Besov spaces $B^p$ and $B^q$
(\cite{BFV}).
\item
Bonet and Vukoti\'c:
$H^\infty _{v_1}$ and $H^\infty_{v_2}$ for general classes of weight functions $v_1,v_2$ (\cite{BV}).
\item
Buckley and Vukoti\'c: the Dirichlet space $\mathcal D$ and the Hardy space $H^p$ (\cite{BuV}).
\item
\'Alvarez, M\'arquez and Vukoti\'c: the Bloch
space $\cb$ and the Bergman space $A^p$ (\cite{AMV}).
\end{itemize}

In \cite{GM} the second and third authors of this paper
characterized the superposition operators between the space $BMOA$
and the Hardy spaces $H^p$ and those between the Bloch space $\cb$
and $H^p$. Now we study the same questions considering the spaces
$\Dp$ in the place of $H^p$, and we compare the corresponding
results.

The space $BMOA$ consists of those functions $f\in H^1$
whose boundary values function
has bounded mean oscillation, that is,
lies in $BMO(\T)$.
The Bloch space $\cb$ is the space of all functions $f\in\hol(\D)$ for which
\[\|f\|_\cb=|f(0)|+\sup_{z\in\D}(1-|z|^2)|f'(z)|<\infty.\]
We have the inclusions
\[BMOA\subset\cb,\qquad
H^\infty\subset BMOA\subset\bigcap_{0<p<\infty}H^p.\]
See \cite{G:BMOA} for the results about the space $BMOA$ and
\cite{ACP} for the Bloch space.

We proved the following (see \cite{GM}, where we actually
considered not only the spaces $BMOA$ and $\cb$ but also
the more general spaces $Q_s$).
\begin{other}\label{BMOABlochHp}
Let $\varphi$ be an entire function and $0<p<\infty$. Then
\begin{itemize}
\item[(a)]
$S_\varphi(BMOA)\subset H^p$ if and
only if $\varphi$ is of order less than one, or of order one and
type zero.
\item[(b)]
$S_\varphi(\cb)\subset H^p$ if and only if
$\varphi$ is constant.
\item[(c)]
$S_\varphi(H^p)\subset BMOA$ if and only if
$\varphi$ is constant.
\item[(d)]
$S_\varphi(H^p)\subset\cb$ if and only if
$\varphi$ is constant.
\end{itemize}
\end{other}

Let us recall that the order of a non-constant entire function
$\varphi$ is
\[\rho=\limsup_{r\to\infty}\frac{\log\log M(r,\varphi)}{\log r},\]
where
$M(r,\varphi)=\max_{|z|=r}|\varphi(z)|$. If $0<\rho<\infty$, the type of the entire function $\varphi$ of order $\rho$ is
\[\tau=\limsup_{r\to\infty}\frac{\log M(r,\varphi)}{r^\rho}.\]
See \cite{Bo} for results about entire functions.

Our main result is the following.

\begin{theorem}\label{BMOABlochDp}
Let $\varphi$ be an entire function and $0<p<\infty$. Then
\begin{itemize}
\item[(a)]
If $p\ge2$, then $S_\varphi(BMOA)\subset\Dp$ if and
only if $\varphi$ is of order less than one, or of order one and
type zero.
If $p<2$, then $S_\varphi(BMOA)\subset\Dp$ if and
only if $\varphi$ is constant.
\item[(b)]
$S_\varphi(\cb)\subset\Dp$ if and only if
$\varphi$ is constant.
\item[(c)]
$S_\varphi(\Dp)\subset BMOA$ if and only if
$\varphi$ is constant.
\item[(d)]
$S_\varphi(\Dp)\subset\cb$ if and only if
$\varphi$ is constant.
\end{itemize}
\end{theorem}

Let us remark that, in the case $p<2$,
the result we obtain here
for the superposition operator from $BMOA$ to $\Dp$
do not coincide with the result we had for the Hardy spaces.

As mentioned before, in \cite{GM} we considered $Q_s$ spaces. These
include the space $BMOA$ (for $s=1$) and the Bloch space (for
$s>1$). In the present work we also consider
superposition operators between $Q_s$ spaces and $\Dp$ spaces
for all the cases ($0\le s<\infty$, $0<p<\infty$).

The next natural question we study is what happens if we take
the spaces $\cd^p_\alpha$ for other values of $p$ and $\alpha$,
not only $\alpha=p-1$.

\section{Proof of Theorem~\ref{BMOABlochDp}}\label{proofthm1}

The result given in (a) of Theorem~\ref{BMOABlochHp}
is also true for the Bergman spaces $A^p$ (see \cite{AMV}).
Now, let $\varphi$ be an entire function and suppose that $2\le p<\infty$. Then we have that $H^p\subset\Dp$
(recall that $H^2=\mathcal D^2_1$ and see \eqref{1}).
It is also true that $\Dp\subset A^{2p}$.

Using these two results about superposition operators
and the fact that $H^p\subset\Dp\subset A^{2p}$,
it follows easily that
$S_\varphi(BMOA)\subset\Dp$ if and
only if $\varphi$ is of order less than one, or of order one and type zero.

For the second part of (a), suppose that $0<p<2$, $\varphi$ is a non-constant entire
function, and $S_\varphi(BMOA)\subset\Dp$ to get a contradiction.
\par
Take a function $f\in H^\infty$ such that
\begin{equation}\label{3}
\int_0^1(1-r)^{p-1}|f'(re^{i\theta})|^p\,dr
=\infty,\quad\text{for almost every $\theta$}.
\end{equation}
The existence of this function is proved by Girela in \cite{Gi92}.

Now, it is clear that $\varphi'\circ f\in H^\infty$.
We have also that $\varphi'\circ f\not\equiv0$.
Indeed, if $\varphi'\circ f\equiv0$ then
we have that $\varphi'\equiv0$ in $f(\D)$, which is an open set
in $\C$ (observe that $f$ is non-constant). This implies that $\varphi'\equiv0$, but we
are assuming that $\varphi$ is non-constant.
Then, as $\varphi'\circ f\in H^\infty$ and
$\varphi'\circ f\not\equiv0$,
it follows that $\varphi'\circ f$ has a non-zero
radial limit almost everywhwere.

Using the fact that $H^\infty\subset BMOA$, we
have that $f\in BMOA$ and, hence, $S_\varphi(f)\in\Dp$, that is,
$$\int_{\D}(1-|z|^2)^{p-1}|(\varphi\circ f)'(z)|^p\,dA(z)
=
\int_{\D}(1-|z|^2)^{p-1}|f'(z)|^p\,|\varphi'(f(z)|^p\,dA(z)<\infty.$$
Equivalently,
$$\int_0^{2\pi}\int_0^1(1-r)^{p-1}|f'(re^{i\theta})|^p\,
|\varphi'(f(re^{i\theta}))|^p\,dr\,d\theta<\infty.$$
This implies that
$$\int_0^1(1-r)^{p-1}|f'(re^{i\theta})|^p\,
|\varphi'(f(re^{i\theta}))|^p\,dr<\infty,\quad\text{for almost every $\theta$}.$$
But, since $\varphi'\circ f$ has a non-zero
radial limit almost everywhwere, we deduce that
$$\int_0^1(1-r)^{p-1}|f'(re^{i\theta})|^p\,dr<\infty,\quad
\text{for almost every $\theta$},$$
which contradicts \eqref{3}.
Then, we have proved (a).

For the proof of (b), let $\varphi$ be an entire function
and $0<p<\infty$. Recall that
$S_\varphi(\cb)\subset H^p$ if and only if $\varphi$ is constant (Theorem~\ref{BMOABlochHp}, (b)).
Now, if $p\le 2$ we have \eqref{2} that $\Dp \subset H^p$.
This implies that
$S_\varphi(\cb)\subset\Dp$ if and
only if $\varphi$ is constant.
\par
Now suppose $p>2$.
In the proof of (b) in Theorem~\ref{BMOABlochHp}
we used the fact that the zero sequence of a function
in $H^p$ satisfies the Blaschke condition. This
is not true for all the functions in $\Dp$ in this case, so
we need more precise results about the zero sequences of Bloch functions and $\Dp$-functions.
\par
Let us assume that $S_\varphi(\cb)\subset\Dp$.
Take a function
$f\in\cb$ with $f(0)\ne0$ such that, if $\{a_n\}_{n=1}^\infty$
is the sequence of zeros of $f$, repeated according to multiplicity, and ordered so that
$|a_1|\le|a_2|\le\dots$, then
\begin{equation}\label{5}
\prod_{n=1}^N\frac1{|a_n|}\ne\op\left((\log N)^{\frac12}\right),\qquad\text{as $N\to\infty$}.
\end{equation}
M.~Nowak \cite[Theorem~1]{N} proved the existence of this function,
a fact which shows the sharpness of Theorem~2~(ii) of \cite{GNW}.
Let us consider the entire function $\varphi_1=\varphi-\varphi(0)$
and the function $g=\varphi_1\circ f$, analytic in the disc. We have
two possibilities:

If $g\equiv0$, then $\varphi_1(f(z))=0$, for all $z\in\D$. Now, $f(\D)$
is an open set, as $f$ is analytic and non-constant in $\D$. Then, the entire function $\varphi_1$ is zero in an open set,
which implies that $\varphi_1$ is identically zero. Then $\varphi$ is constant.

If $g\not\equiv0$, using that $S_\varphi f=\varphi\circ f\in\Dp$,
we see that the function $g$, given by
\[g(z)=\varphi_1(f(z))=\varphi(f(z))-\varphi(0),\qquad z\in\D,\]
also belongs to $\Dp$, and, by a result proved in \cite{GP:Aust06},
its sequence of ordered non-null zeros $\{z_k\}_{k=1}^\infty$ satisfies that
\begin{equation}\label{4}
\prod_{k=1}^N\frac1{|z_k|}
=\og\left((\log N)^{\frac12-\frac 1p}\right),\qquad\text{as $N\to\infty$}.
\end{equation}
Now observe that for all $n$ we have
$g(a_n)=\varphi(f(a_n))-\varphi(0)=\varphi(0)-\varphi(0)=0$, that is,
all the zeros $a_n$ of $f$ are also zeros of $g$. Besides, it is easy
to see that, for all $n$, the multiplicity of $a_n$ as zero of $g$ is greater
than or equal to its multiplicity as zero of $f$. Then it is clear
that $|a_n|\ge|z_n|$ for all $n$, and
\[\prod_{n=1}^N\frac1{|a_n|}\le\prod_{n=1}^N\frac1{|z_n|},\qquad\text{for all $N$.}\]
Now, by \eqref{4}, there exist $C>0$ and $N_0\in\N$ such that
\[\prod_{k=1}^N\frac1{|z_k|}
\le C(\log N)^{\frac12-\frac 1p},\qquad\text{for all $N>N_0$.}\]
Then, for all $N>N_0$ we have
\[\frac{\prod_{n=1}^N\frac1{|a_n|}}{(\log N)^{\frac12}}
\le\frac{\prod_{n=1}^N\frac1{|z_n|}}{(\log N)^{\frac12}}
\le\frac{C(\log N)^{\frac12-\frac1p}}{(\log N)^{\frac12}}
=\frac C{(\log N)^{\frac1p}},\]
which tends to zero, as $N$ tends to infinity. This implies that
\[\prod_{n=1}^N\frac1{|a_n|}=\op\left((\log N)^{\frac12}\right),\qquad\text{as $N\to\infty$},\]
contradiction with \eqref{5}. Then, this possibility ($p>2$) has to
be excluded. We deduce that $\varphi$ must be constant, and the
proof of (b) is finished.

Using the fact that $BMOA\subset\cb$, we see that (d) implies (c), so it remains to prove (d). Let us remark that, if $0<p<\infty$, $\alpha>-1$
and $\beta>0$, then the analytic function
\[f(z)=\frac1{(1-z)^\beta},\qquad z\in\D,\]
belongs to the Bergman space
$A^p_\alpha$ if and only if $\beta<\frac{2+\alpha}p$.

Suppose that $\varphi$ is an entire function, $0<p<\infty$
and
$S_\varphi(\Dp)\subset\cb$. Take $\alpha\in(0,\frac1p)$ and let
\[f(z)=\frac1{(1-z)^\alpha},\qquad z\in\D.\]
Using the previous condition
for the function $f'$
(taking $\alpha=p-1$, $\beta=\alpha+1$),
we can see that $f\in\Dp$. It is
also easy to check that $f\notin\cb$.

The idea of the proof is that functions in the Bloch space
have logarithmic growth, but $f$ does not
have this growth and neither does the function
$S_\varphi(f)=\varphi\circ f$,
unless $\varphi$ is constant.
Then $S_\varphi(f)\not\in\cb$, that leds to contradiction.

For a detailed proof, we can
use this function $f$ and follow the steps of the proof of Proposition~1
in~\cite{AMV}, where the authors proved this result
for the Bergman space $A^p$ in the place of the space $\Dp$.

This finishes the proof of Theorem~\ref{BMOABlochDp}.

\section{Superposition between $Q_s$ and $\Dp$ spaces}

If $0\le s<\infty$, we say that $f\in Q_s$ if $f$ is analytic in $\D$ and
\[\sup_{a\in\D}\int_\D|f'(z)|^2g(z,a)^s\,dA(z)<\infty,\]
where $g(z,a)$ is the Green's function in $\D$, given by
$g(z,a)=\log\left|\frac{1-\overline{a}z}{z-a}\right|$.

These spaces arose in~\cite{AL} in connection with Bloch and normal
functions and have been studied by several authors (see e.g.
\cite{AGW, AL, ASX, AXZ, DG, EX, NX, PP, X-paci}). Aulaskari and
Lappan \cite{AL} proved that, for $s>1$, $Q_s$ is the Bloch space
$\mathcal B$. Using one of the many characterizations of the space
$BMOA$ (see, e.g., Theorem 5 of \cite{Ba}) we can see that
$Q_1=BMOA$. If $s=0$, just by the definitions, $Q_s$ is the
classical Dirichlet space $\mathcal D$.
\par
It is well known that $\mathcal D\subset BMOA$ and $BMOA\subset\mathcal B$.
R.~Aulaskari, J.~Xiao and R.~Zhao proved in~\cite{AXZ} that
\[\mathcal D\subset Q_{s_1}\subset Q_{s_2}\subset BMOA,\qquad0<s_1<s_2<1.\]
The books \cite{X01} and \cite{X06} are general references for the theory of $Q_s$-spaces.

We have considered superposition operators between $Q_s$ spaces and
spaces of Dirichlet type $\Dp$ in the particular cases
$s=1$, $s>1$, and now we consider the remaining cases.
If $0<s<1$ we have the following result.

\begin{theorem}\label{QsDp}
Let $\varphi$ be an entire function, $0<s<1$ and $0<p<\infty$. Then
\begin{itemize}
\item[(a)]
$S_\varphi(Q_s)\subset\Dp$ if and
only if $\varphi$ is of order less than one, or of order one and type zero.
\item[(b)]
$S_\varphi(\Dp)\subset Q_s$ if and only if
$\varphi$ is constant.
\end{itemize}
\end{theorem}

\begin{proof}
Let us mention that the authors proved (a) for the Hardy spaces
$H^p$ in place of $\Dp$ \cite[Theorem~1]{GM}. On the other hand,
we have that
$S_\varphi(\cb)\subset A^p$ if and only if
$\varphi$ is of order less than one, or of order one and type zero
\cite[Theorem~3]{AMV}.
We note that the Bloch function used there is univalent, then it also belongs to the
space $Q_s$, $0<s<1$. So
we can adapt that proof to see that (a) is
also true for the Bergman spaces $A^p$.

Now the proof of (a) is clear
if $p\ge2$, having in mind that
$H^p\subset\Dp\subset A^{2p}$. It remains to
prove (a) for $p<2$.

Then, let us assume that $\varphi$ is an entire function, $0<s<1$ and $0<p<2$.
The first part is easily deduced. Indeed,
if $S_\varphi(Q_s)\subset\Dp$ then $S_\varphi(Q_s)\subset H^p$ (see \eqref{2}). Applying that (a) is
true for $H^p$ as mentioned above, we deduce that
$\varphi$ is of order less than one, or of order one and type zero.
For the second part, suppose that
$\varphi$ is of this order and type, and
take $f\in Q_s$. We have to
prove that $S_\varphi(f)\in\Dp$, that is, the integral
\[\int_\D(1-|z|^2)^{p-1}\bigl|\bigl(S_\varphi(f)\bigr)'(z)\bigr|^p\,dA(z)\]
is finite. Applying H\"older's inequallity for $\frac2p>1$ and $\frac2{2-p}$, we have
\begin{align}
&\int_\D(1-|z|^2)^{p-1}\bigl|\bigl(S_\varphi(f)\bigr)'(z)\bigr|^p\,dA(z)
=\int_\D(1-|z|^2)^{p-1}|f'(z)|^p|\varphi'(f(z))|^p\,dA(z)\notag\\[5pt]
&=\int_\D\left((1-|z|^2)^{\frac{sp}2}|f'(z)|^p\right)
\left((1-|z|^2)^{\frac{p(2-s)-2}2}
|\varphi'(f(z))|^p\right)\,dA(z)\notag\\[5pt]\label{2int}
&\le\left(\int_\D(1-|z|^2)^s|f'(z)|^2\,dA(z)\right)^{\frac p2}
\left(\int_\D(1-|z|^2)^{\frac{p(2-s)-2}{2-p}}
|\varphi'(f(z))|^{\frac{2p}{2-p}}\,dA(z)\right)^{\frac{2-p}2}.
\end{align}

Now, since $f\in Q_s$ we have that the measure $\mu$ defined by $d\mu(z)=(1-|z|)^s|f'(z)|^2\,dA(z)$ is finite (in fact, $\mu$ is a $s$-Carleson measure \cite[Theorem 1.1]{ASX}). So, we have that
\begin{equation}\label{sCarleson}
\int_\D(1-|z|)^s|f'(z)|^2\,dA(z)<\infty.\end{equation}
Then, the first integral of \eqref{2int}
is finite, and now we consider the second
integral.
Let $\alpha=\frac{p(2-s)-2}{2-p}$, the exponent of $(1-|z|^2)$.
Observe that $\alpha>-1$, and
\begin{equation}\label{csubalpha}
(1-|z|^2)^\alpha=(1+|z|)^\alpha(1-|z|)^\alpha\le c_\alpha(1-|z|)^\alpha,\qquad\text{for all $z\in\D$,}
\end{equation}
for a positive constant $c_\alpha$ which only depends on $\alpha$.
Using that $Q_s\subset\cb$, we know that $f\in\cb$, and we have
\begin{equation}\label{crecBl}
|f(z)|\le\left(\log\frac1{1-|z|}+1\right)\|f\|_\cb,
\qquad\text{for all $z\in\D$.}
\end{equation}
Let $c=\|f\|_\cb+1$ and take a positive number $\varepsilon$ such that
$\varepsilon<\frac{(1+\alpha)(2-p)}{2pc}$.
Recalling that the derivative $\varphi'$ of
$\varphi$ has the same order and type of
$\varphi$, we see that $\varphi'$
is of order less than one, or of order one and type zero. Then, there
exists $A>0$ such that
\begin{equation}\label{phipr}
|\varphi'(z)|\le A\,e^{\varepsilon|z|},\qquad
\text{for all $z\in\C$.}
\end{equation}

Using \eqref{csubalpha}, \eqref{phipr} and \eqref{crecBl}, we have that
\begin{align*}
\int_\D(1-|z|^2)^\alpha|\varphi'(f(z))|^{\frac{2p}{2-p}}\,dA(z)
&\le c_\alpha\int_\D(1-|z|)^\alpha\left(A\,e^{\varepsilon|f(z)|}\right)^{\frac{2p}{2-p}}dA(z)\\[5pt]
&\le c_\alpha\,A^{\frac{2p}{2-p}}\int_\D(1-|z|)^\alpha
\left(e^{\varepsilon c\left(\log\frac1{1-|z|}+1
\right)}\right)^{\frac{2p}{2-p}}dA(z)\\[5pt]
&\le c_\alpha\,A^{\frac{2p}{2-p}}\int_\D(1-|z|)^\alpha
\left(\frac{e}{1-|z|}\right)^{\frac{2p}{2-p}\varepsilon c}dA(z)\\[5pt]
&\le c_\alpha\left(A\,e^{\varepsilon c}\right)^{\frac{2p}{2-p}}
\int_\D(1-|z|)^{\alpha-\frac{2p}{2-p}\varepsilon c}\,dA(z).
\end{align*}
The last integral is finite if $\alpha-\frac{2p}{2-p}\varepsilon c>-1$, which is true by our choice of $\varepsilon$.
Then, the second integral of \eqref{2int} is also finite, and we have that
$S_\varphi(f)\in\Dp$, that finishes the proof of (a).

We deduce (b) from Theorem \ref{BMOABlochDp}, (d) and the inclusion $Q_s\subset\cb$.
\end{proof}

For $s=0$, the result about
superposition operators between $Q_s$ and $\Dp$ is actually easily deduced from previous theorems.

\begin{theorem}\label{DirichletDp}
Let $\varphi$ be an entire function and $0<p<\infty$. Then
\begin{itemize}
\item[(a)]
$S_\varphi(\cd)\subset\Dp$ if and
only if $\varphi$ is of order less than $2$, or of order $2$ and finite type.
\item[(b)]
$S_\varphi(\Dp)\subset\cd$ if and only if
$\varphi$ is constant.
\end{itemize}
\end{theorem}

\begin{proof}
First, let us notice that (a) was proved for Bergman spaces $A^p$ and
Hardy spaces $H^p$ (see Theorem~1 for $p=2$ and Remark 8 of \cite{BuV}).
Then, if $p\ge2$, using that $H^p\subset\Dp\subset A^{2p}$, we see
that this result is also true for $\Dp$.
On the other hand, if $p<2$, (a) is proved in Theorem~24 of \cite{BFV} ($p=2$, $\beta=q-1$).
Now, if $S_\varphi(\Dp)\subset\cd$,
using the inclusion $\cd\subset\cb$, we can
apply (d) of Theorem~\ref{BMOABlochDp} to prove (b).
\end{proof}

\section{Superposition between $Q_s$ and $B^p$ spaces}

We have characterized superposition operators between $Q_s$ and
$\cd^p_\alpha$ spaces in the particular case $\alpha=p-1$.
Our next purpose is considering other values of $p$ and $\alpha$.
Let us begin studing what happens if $\alpha=p-2$. As we have mentioned
before, for $p>1$, the space $\cd^p_{p-2}$ is the Besov space
$B^p$, that is, the space of those
functions $f$, analytic in $\D$, for which
\[\int_\D(1-|z|^2)^{p-2}|f'(z)|^p\,dA(z)<\infty.\]
For $p=2$, we have $B^2=\cd^2_0=\cd$.
The space $B^1$ is separately defined as the space of those
analytic functions $f$ in $\D$ such that
\[\int_\D|f''(z)|\,dA(z)<\infty.\]
The following inclusions hold:
\begin{align*}
&B^p\subset B^q,&&\text{if $1<p<q<\infty$,}\\
&B^1\subset B^p\subset BMOA,&&\text{if $1<p<\infty$,}\\
&B^1\subset H^\infty.&&
\end{align*}

S.M. Buckley, J.L. Fern\'andez and D. Vukoti\'c proved in \cite{BFV} the
following results:
\begin{other}\label{B1}
For any entire function $\varphi$, we have $S_\varphi(B^1)\subset B^1$.
\end{other}

\begin{other}\label{BpBloch}
Let $1<p<\infty$. If $\varphi$ is an entire function, then
$S_\varphi(B^p)\subset\mathcal B$ if and only if $\varphi$ is a linear function,
$\varphi(z)=az+b$ ($a,b\in\C$).
\end{other}

\begin{otherc}\label{BpBq}
Let $\varphi$ be an entire function.
\begin{itemize}\itemsep3pt
\item[(a)]
If $\,1<p\le q<\infty$, then $S_\varphi(B^p)\subset B^q$ if and only if $\varphi$
is linear.
\item[(b)]
If $\,1\le q<p<\infty$, then $S_\varphi(B^p)\subset B^q$ if and only $\varphi$ is constant.
\end{itemize}
\end{otherc}

Observe that $\mathcal B=Q_s$ for $s>1$ and $B^2=\mathcal D=Q_0$.
Then, some particular cases of our problem are already solved in the
previous results. Indeed, Theorem~\ref{BpBloch} characterize the superposition
operators from $B^p$ into $Q_s$ for $1<p<\infty$ and $1<s<\infty$, and
taking $p=2$ or $q=2$, Corollary~\ref{BpBq} characterize the superposition
operators from $Q_0$ into $B^p$ for $1\le p<\infty$ and
from $B^p$ into $Q_0$ for $1<p<\infty$.

We have the following result.
\begin{theorem}\label{BpQs}
Let $\varphi$ be an entire function. Then
\begin{itemize}\itemsep3pt
\item[(a)]
For $0\le s<\infty$, $S_\varphi(B^1)\subset Q_s$.
\item[(b)]
For $1<p<\infty$ and $0<s\le1$, $S_\varphi(B^p)\subset Q_s$ if and only if
\begin{itemize}
\item
$\varphi$ is constant, if $\,\frac1p\le\frac{1-s}2$,
\item
$\varphi$ is a linear function, if $\,\frac1p>\frac{1-s}2$.
\end{itemize}
\item[(c)]
For $0<s<\infty$ and $1\le p<\infty$,
$S_\varphi(Q_s)\subset B^p$ if and only if $\varphi$ is constant.
\end{itemize}
\end{theorem}

Let us remark that Theorem~\ref{BpQs}, together with
Theorem~\ref{BpBloch} and Corollary~\ref{BpBq},
give a complete characterization of the
superposition operators between $Q_s$ spaces and $B^p$ spaces for
all the cases ($0\le s<\infty$, $1\le p<\infty$).

We shall use the following result in the proof of Theorem~\ref{BpQs}.

\begin{proposition}\label{BpcontQs}
If \,$0<s\le1$ and $1<p<\infty$, then
\[B^p\subset Q_s\iff\tfrac1p>\tfrac{1-s}2.\]
\end{proposition}

\begin{proof}
As mentioned before, $B^p\subset BMOA$ if
$1<p<\infty$. Then, the result is true for $s=1$, and we may assume that $s<1$.

First, suppose that $0<s<1$ and $1<p\le2$.  Then $\frac1p\ge\frac12>\frac{1-s}2$, and
\[B^p\subset B^2=\mathcal D=Q_0\subset Q_s.\]

Secondly, assume that $0<s<1$, $2<p<\infty$ and $\frac1p>\frac{1-s}2$.

It is not difficult to check that, for $1\le p<\infty$, the Besov space $B^p$ is contained
in the Lipschitz space $\Lambda_{1/p}^p$, that is:
\begin{equation}\label{lipschitz}
f\in B^p\implies M_p(r,f')=\og\left(\tfrac1{(1-r)^{1-1/p}}\right),
\qquad\text{as $r\to1$.}
\end{equation}

On the other hand, for $p,s$ verifying the previous conditions,
it is known that $\Lambda_{1/p}^p\subset Q_s$
(see \cite{ASX}). Then $B^p\subset\Lambda_{1/p}^p\subset Q_s$.

Finally, suppose that $0<s<1$, $2<p<\infty$  and $\frac1p\le\frac{1-s}2$. In order to prove that
$B^p\not\subset Q_s$, let us take the function
\[f(z)=\sum_{k=1}^\infty\frac1{\sqrt k\,2^{\frac kp}}\,z^{2^k},\qquad z\in\D.\]
Observe that
\[\sum_{k=1}^\infty\frac1{\sqrt k\,2^{\frac kp}}
\le\sum_{k=1}^\infty\frac1{2^{\frac kp}}
=\sum_{k=1}^\infty\left(2^{-\frac1p}\right)^k<\infty.\]
Then, we see that $f$ is an analytic function in the disc
given by a power series with Hadamard gaps.

Recall that a power series of the form $\sum_{k=0}^\infty
a_kz^{n_k}$, where $a_k\in\C$ for all $k$, and $\{n_k\}$ is a
strictly increasing sequence of non-negative integers, such that
\[\frac{n_{k+1}}{n_k}\ge\lambda,\qquad k=1,2,3,\dots\,,\]
for some $\lambda>1$,
is called a lacunary series or a power series with Hadamard gaps.
\par
Using the characterization of the lacunary series in Besov spaces,
given in Theorem D of~\cite{DGV},
\begin{equation}\label{caratBp}
f\in B^p\iff\sum_{k=0}^\infty n_k|a_k|^p<\infty,
\end{equation}
we see that
\[\sum_{k=0}^\infty n_k|a_k|^p
=\sum_{k=1}^\infty2^k\left(\frac1{\sqrt k\,2^{\frac kp}}\right)^p
=\sum_{k=1}^\infty\left(\frac1{\sqrt k}\right)^p
=\sum_{k=1}^\infty\frac1{k^\frac p2}<\infty,\]
since $p>2$. Then $f\in B^p$.
Now, we use the characterization of the lacunary series in $Q_s$ spaces
for $0<s<1$ (see Theorem 6 of \cite{AXZ}),
\begin{equation}\label{qslacunar}
f\in Q_s\iff\sum_{k=0}^\infty2^{k(1-s)}\sum_{j:\,n_j\in I_k}|a_j|^2<\infty,
\end{equation}
where $I_k=\{n\in\N:2^k\le n<2^{k+1}\}$ for all $k$.
We have that
\begin{align*}
\sum_{k=0}^\infty2^{k(1-s)}\sum_{j:n_j\in I_k}|a_j|^2
&=\sum_{k=0}^\infty2^{k(1-s)}|a_k|^2
=\sum_{k=1}^\infty2^{k(1-s)}\frac1{k\,2^{\frac{2k}p}}\\[4pt]
&=\sum_{k=1}^\infty2^{k(1-s-\frac2p)}\frac1k
\ge\sum_{k=1}^\infty\frac1k=\infty.
\end{align*}
Here, we used the inequality $\frac1p\le\frac{1-s}2$ to
see that $2^{k(1-s-\frac2p)}\ge1$ for all $k$.
By~\eqref{qslacunar}, we have that $f\notin Q_s$. We have finished the proof,
since the function $f$ belongs to $B^p\setminus Q_s$.
\end{proof}

\begin{proof}[Proof of Theorem~\ref{BpQs}]
Part (a) follows trivially from Theorem~\ref{B1}. Indeed, if $0\le s<\infty$, we have
\[S_\varphi(B^1)\subset B^1\subset B^2=\mathcal D=Q_0\subset Q_s.\]

In order to prove part (b), suppose that $1<p<\infty$ and $0<s\le1$. If
$S_\varphi(B^p)\subset Q_s$, then $S_\varphi(B^p)\subset\mathcal B$, and
it follows from Theorem \ref{BpBloch} that $\varphi$ is linear, that is,
$\varphi(z)=az+b$, for two constants $a,b\in\C$.

If in addition, $\frac1p\le\frac{1-s}2$,
then, by Proposition~\ref{BpcontQs}, $B^p\not\subset Q_s$, so there exists a function
$f\in B^p\setminus Q_s$. As we have assumed that $S_\varphi(B^p)\subset Q_s$,
we have that $S_\varphi f=af+b\in Q_s$. Now, $a\ne0$ implies that $f\in Q_s$,
which is a contradiction. It follows that $a=0$ and $\varphi$ is constant.

For the second implication, it is enough to observe that in the case $\frac1p>\frac{1-s}2$
we have the inclusion $B^p\subset Q_s$ given in Proposition~\ref{BpcontQs}, so it is clear that
$S_\varphi(B^p)\subset Q_s$ if $\varphi$ is linear. In this way we see that (b) is proved.

Now, assume that $0<s<\infty$ and $1\le p<\infty$.
If $S_\varphi(Q_s)\subset B^p$ then
$S_\varphi(\mathcal D)\subset B^p$, because $\mathcal D\subset Q_s$.
Then $\varphi$ is linear, by Corollary~\ref{BpBq} in the case $p=2$.
Now, it is known (see Theorem 4.2.1 of \cite{X01}) that there exists a function
\[f\in\bigcap_{0<s<\infty}Q_s\setminus\bigcup_{1\le p<\infty}\Lambda^p_{1/p}.\]
Using again \eqref{lipschitz}, that is, $B^p\subset\Lambda^p_{1/p}$,
we see that $f\in Q_s\setminus B^p$,
so we deduce that $\varphi$ is constant, using the same argument that in part (b).
\end{proof}

\section{Superposition between $Q_s$ and $A^p_\alpha$ spaces}

Our next aim is characterizing
superposition operators between $Q_s$
and $\Dpa$ spaces, when $\alpha>p-1$.
Let us recall that
$\Dpa=A^p_{\alpha-p}$, if $p>0$ and $\alpha>p-1$.
Then we consider superposition between $Q_s$ spaces ($0\le s<\infty$)
and weighted Bergman spaces $A^p_\alpha$ ($p>0$, $\alpha>-1$).

We have the following result.

\begin{theorem}
Let $\varphi$ be an entire function.
Let $0<p<\infty$ and $\alpha>-1$. Then
\begin{itemize}
\item[(a)]
$S_\varphi(\mathcal D)\subset A^p_\alpha$
if and only if
$\varphi$ is of order less than $2$, or of order $2$ and finite type.
\item[(b)]
For $0<s<\infty$,
$S_\varphi(Q_s)\subset A^p_\alpha$ if
and only if $\varphi$ is of order less than one, or of order one and
type zero.
\item[(c)]
For $0\le s<\infty$,
$S_\varphi(A^p_\alpha)\subset Q_s$ if and only if $\varphi$ is constant.
\end{itemize}
\end{theorem}

\begin{proof}

The first implication of part (a) can be obtained adapting
the proof of the following result
(see Theorem~1 of \cite{BuV}).

\begin{other}
Suppose $1<p<\infty$ and $0<q<\infty$. If $\varphi$ is an entire function, then
$S_\varphi(B^p)\subset A^q$ if
and only if $\varphi$ is of order less than $\frac p{p-1}$, or of order
$\frac p{p-1}$ and finite type.
\end{other}

For the second implication, let us assume that $\varphi$ is an entire function
of order less than $2$, or of order $2$ and finite type, let $0<p<\infty$
and $\alpha>-1$. Let's take a function $f\in\mathcal D$ and we shall prove that
$S_\varphi(f)\in A^p_\alpha$.
We know, because of the order and type of $\varphi$, that there exist $C,A>0$ such that
\[|\varphi(w)|\le C\,e^{A|w|^2},\qquad\text{for all $w\in\C$.}\]
Then, we have
\begin{align}
&\int_\D(1-|z|^2)^\alpha|S_\varphi(f)(z)|^p\,dA(z)
=\int_\D(1-|z|^2)^\alpha|\varphi(f(z))|^p\,dA(z)\notag\\[3pt]
&\le\int_\D(1-|z|^2)^\alpha\left(C\,e^{A|f(z)|^2}\right)^pdA(z)
=C^p\int_\D(1-|z|^2)^\alpha e^{Ap|f(z)|^2}dA(z).\label{inteap}
\end{align}

Now, $f$ belongs to the Dirichlet space, so it satisfies
\[|f(z)|=\op\left(\left(\log\frac1{1-|z|}\right)^{1/2}\right),
\qquad\text{as $|z|\to1$.}\]
Then, if we take a constant $\varepsilon$ satisfying
\[0<\varepsilon<\left(\frac{1+\alpha}{Ap}\right)^{1/2},\]
there exists $R_0\in(0,1)$ such that
\[|f(z)|\le\varepsilon\left(\log\frac1{1-|z|}\right)^{1/2},
\quad\text{if $R_0\le|z|<1$.}\]
It follows that, if $R_0\le|z|<1$, we have
\[|f(z)|^2\le\varepsilon^2\log\frac1{1-|z|},
\quad e^{Ap|f(z)|^2}\le e^{Ap\varepsilon^2\log\frac1{1-|z|}}
=\frac1{(1-|z|)^{Ap\varepsilon^2}}.\]

In order to prove that $S_\varphi(f)\in A^p_\alpha$,
it is enough to see that the second integral in \eqref{inteap} is finite, which is equivalent to
\begin{equation}\label{6}
\int_{\{z\in\D:R_0\le|z|<1\}}(1-|z|^2)^\alpha e^{Ap|f(z)|^2}\,dA(z)<\infty.
\end{equation}
We have that
\begin{align*}
\int_{\{z\in\D:R_0\le|z|<1\}}
\hspace{-20pt}
(1-|z|^2)^\alpha\,e^{Ap|f(z)|^2}\,dA(z)
&\le\int_{\{z\in\D:R_0\le|z|<1\}}(1-|z|^2)^\alpha\,\frac1{(1-|z|)^{Ap\varepsilon^2}}\,dA(z)\\[5pt]
&=\int_{\{z\in\D:R_0\le|z|<1\}}(1+|z|)^\alpha\,\frac1{(1-|z|)^{Ap\varepsilon^2-\alpha}}\,dA(z),
\end{align*}
which is finite if $Ap\varepsilon^2-\alpha<1$.
We have chosen $\varepsilon$ satisfying this condition, so \eqref{6} is true and, therefore, $S_\varphi(f)\in A^p_\alpha$.
Then, we have finished the proof of (a).

In the proof of Theorem~\ref{QsDp}, we saw a particular case
($\alpha=0$, $0<s<1$) of part (b), adapting a proof of
\cite{AMV}. Now we just notice that,
taking any $\alpha>-1$ instead of $\alpha=0$,
this adaptation is in fact possible to prove (b).

About part (c), we see that $A^p_\alpha=\cd^p_{p+\alpha}$, and
using \eqref{DpalphaDpbeta}
we can compare this space with $\Dp$, already studied,
so we have that $\Dp\subset D^p_{p+\alpha}$. Now, if
$S_\varphi(A^p_\alpha)\subset Q_s$ then
$S_\varphi(\Dp)\subset Q_s\subset\cb$,
and part (d) of Theorem~\ref{BMOABlochDp} says that $\varphi$ is
constant.
\end{proof}

\section{Superposition between $Q_s$ and $\cd^p_\alpha$ spaces, $\alpha<p-2$}

Next, let us consider the spaces $\cd^p_\alpha$ in the case $\alpha<p-2$.
The functions in these spaces of Dirichlet type are bounded, that is:
\[\cd^p_\alpha\subset H^\infty,\qquad\text{if $\alpha>-1$, $p>\alpha+2$.}\]

Our next result characterize the
superposition operators between $Q_s$ spaces ($0\le s<\infty$)
and $\cd^p_\alpha$ spaces ($\alpha>-1$, $p>\alpha+2$).

\begin{theorem}
Let $\varphi$ be an entire function,
$0\le s<\infty$, $\alpha>-1$ and $p>\alpha+2$.
Then
\begin{itemize}\itemsep5pt
\item[(a)]
$S_\varphi(Q_s)\subset\cd^p_\alpha$ if and only if $\varphi$ is constant.
\item[(b)]
For $s\ge1$, $S_\varphi(\cd^p_\alpha)\subset Q_s$.
\item[(c)]
For $0\le s<1$ and $\frac{1+s}2>\frac{\alpha+1}p$,
$S_\varphi(\cd^p_\alpha)\subset Q_s$.
\item[(d)]
For $0\le s<1$ and $\frac{1+s}2\le\frac{\alpha+1}p$,
$S_\varphi(\cd^p_\alpha)\subset Q_s$ if and only if $\varphi$ is constant.
\end{itemize}
\end{theorem}

Using that $\cd^p_\alpha\subset H^\infty$, we see that
part (a) is a consequence of the following result, proved in \cite{GM}.
\begin{other}
Let $\varphi$ be an entire function. For $0\le s<\infty$,
$S_\varphi(Q_s)\subset H^\infty$ if and only if $\varphi$ is constant.
\end{other}

For part (b), let $1\le s<\infty$, $\alpha>-1$ and $p>\alpha+2$.
If $f\in \cd^p_\alpha$ then $f\in H^\infty$, and we deduce that, for any entire
function $\varphi$, we have
\[S_\varphi(f)=\varphi\circ f\in H^\infty\subset BMOA=Q_1\subset Q_s.\]

We prove parts (c) and (d) in two steps, given in the following results:

\begin{proposition}\label{prop2}
Let $\varphi$ be an entire function,
$0\le s<1$, $\alpha>-1$ and $p>\alpha+2$.
\begin{itemize}\itemsep5pt
\item[(i)]
If $\cd^p_\alpha\subset Q_s$, then
$S_\varphi(\cd^p_\alpha)\subset Q_s$.
\item[(ii)]
If $\cd^p_\alpha\not\subset Q_s$, then
$S_\varphi(\cd^p_\alpha)\subset Q_s$ if and only if $\varphi$ is constant.
\end{itemize}
\end{proposition}

\begin{proposition}\label{prop3}
If $0\le s<1$, $\alpha>-1$ and $p>\alpha+2$, then
\[\cd^p_\alpha\subset Q_s\iff\tfrac{1+s}2>\tfrac{\alpha+1}p.\]
\end{proposition}

\section{Proofs of Proposition \ref{prop2} and Proposition \ref{prop3}}

In order to prove Proposition \ref{prop2}, let $\varphi$ be an entire function, $0\le s<1$, $\alpha>-1$ and $p>\alpha+2$.
Suppose that $\cd^p_\alpha\subset Q_s$. If $f\in \cd^p_\alpha$ then $f\in H^\infty$
and, consequently, $\varphi'\circ f$ is bounded on the disc, that is, there
exists $c>0$ such that
\[|\varphi'(f(z))|\le c,\qquad\text{for all $z\in\D$.}\]
Then, we have
\begin{align*}
\int_\D(1-|z|^2)^\alpha|(\varphi\circ f)'(z)|^p dA(z)
&=\int_\D(1-|z|^2)^\alpha|\varphi'(f(z))|^p|f'(z)|^p dA(z)\\[4pt]
&\le c^p\int_\D(1-|z|^2)^\alpha|f'(z)|^p dA(z)<\infty,
\end{align*}
that is, $S_\varphi f=\varphi\circ f\in \cd^p_\alpha$. Then $S_\varphi f\in Q_s$
and (i) is proved.

For the second part, let
$0\le s<1$, $\alpha>-1$ and $p>\alpha+2$, such that
$\cd^p_\alpha\not\subset Q_s$, and suppose that $\varphi$ is an entire function
satisfying $S_\varphi(\cd^p_\alpha)\subset Q_s$.
Let us assume that $\varphi$ is not constant to get a contradiction.
\par
We have that $\varphi'\not\equiv0$, so there exist
$z_0\in\C$ and $R,A>0$ such that
\[
|\varphi'(z)|\ge A,\qquad\text{if $|z-z_0|\le R$.}
\]
Take a function $f\in \cd^p_\alpha$ such that $f\notin Q_s$.
Then $f\in H^\infty$ and we can take $M>0$ such that $|f(z)|\le M$,
for any $z\in\D$.
Consider the function $h$ defined by
\[h(z)=z_0+\frac RMf(z),\qquad z\in\D.\]
Then $h\in \cd^p_\alpha\setminus Q_s$ and $|h(z)-z_0|\le R$, for any
$z\in\D$, so we have that
\[
|\varphi'(h(z))|\ge A,\qquad\text{for all $z\in\D$.}
\]

Using the assumption $S_\varphi(\cd^p_\alpha)\subset Q_s$, we
deduce that $S_\varphi(h)=\varphi\circ h\in Q_s$, that is,
\[
\sup_{a\in\D}\int_\D|(\varphi\circ h)'(z)|^2g(z,a)^s\,dA(z)<\infty.
\]
Now, for all $a\in\D$, we see that
\[
\begin{aligned}
\int_\D|(\varphi\circ h)'(z)|^2g(z,a)^s\,dA(z)
&=\int_\D|h'(z)|^2|\varphi'(h(z))|^2g(z,a)^s\,dA(z)\\
&\ge A^2\int_\D|h'(z)|^2g(z,a)^s\,dA(z).
\end{aligned}
\]
It follows that
\[\sup_{a\in\D}\int_\D|h'(z)|^2g(z,a)^s\,dA(z)<\infty,\]
which means that $h\in Q_s$, so we get a contradiction
and finish the proof of Proposition~\ref{prop2}.

\bigskip

For the proof of Proposition~\ref{prop3},
let $0\le s<1$, $\alpha>-1$ and $p>\alpha+2$.
We shall consider different cases.

\subsection*{Case 1}

First, suppose $\frac{1+s}2=\frac{\alpha+1}p$. Let us observe that
\[\frac{\alpha+1}p=\frac{1+s}2\ge\frac12,\qquad \alpha+1\ge\frac p2>\frac{\alpha+2}2.\]
It follows that $\alpha>0$ and, therefore, $p>2$.

We are going to prove that $\cd^p_\alpha\not\subset Q_s$ using a function
given by a lacunary series, as we did in the proof of Proposition~\ref{BpcontQs}.
Let us take the function
\[f(z)=\sum_{k=1}^\infty\frac1{\sqrt k\,2^{Ak}}\,z^{2^k},\qquad z\in\D,\qquad\text{where $A=\frac{1-s}2$.}\]
We see that
\[\sum_{k=1}^\infty\frac1{\sqrt k\,2^{Ak}}
\le\sum_{k=1}^\infty\frac1{2^{Ak}}
=\sum_{k=1}^\infty\left(2^{-A}\right)^k<\infty.\]
Then $f$ is an analytic function in the disc
given by a power series with Hadamard gaps $\sum_{k=0}^\infty a_kz^{n_k}$.
Using the characterization of the lacunary series in $A^p_\alpha$ spaces,
given in Proposition 2.1 of \cite{BKV}, we easily deduce that
\begin{equation}\label{Dpalphalac}
f\in \cd^p_\alpha\iff\sum_{k=1}^\infty\frac{|a_k|^p}{n_k^{\alpha+1-p}}<\infty.
\end{equation}
Now we see that
\[\sum_{k=1}^\infty\frac{|a_k|^p}{n_k^{\alpha+1-p}}
=\sum_{k=1}^\infty\frac1{k^{\frac p2}\,\,2^{Akp}\,\,2^{k(\alpha+1-p)}}
=\sum_{k=1}^\infty\frac1{k^{\frac p2}\,2^{k(Ap+\alpha+1-p)}}.\]
Here, we have
\[Ap+\alpha+1-p=\left(A+\frac{\alpha+1}p-1\right)p
=\left(\frac{1-s}2+\frac{1+s}2-1\right)p=0.\]
Then, having in mind that $p>2$, we see that
\[\sum_{k=1}^\infty\frac{|a_k|^p}{n_k^{\alpha+1-p}}
=\sum_{k=1}^\infty\frac1{k^{\frac p2}}<\infty,\]
and it follows that $f\in \cd^p_\alpha$.
On the other hand, if $0<s<1$, using \eqref{qslacunar} we have that
\[f\in Q_s\iff\sum_{k=1}^\infty 2^{k(1-s)}|a_k|^2<\infty.\]
Let us notice that for the function $f$ defined here, this equivalence is also
true if $s=0$. Indeed, using \eqref{Dpalphalac} for $\cd^2_0=\cd=Q_0$, we have that
\[f\in Q_0\iff\sum_{k=1}^\infty n_k|a_k|^2<\infty
\iff\sum_{k=1}^\infty 2^k|a_k|^2<\infty.\]
We see that
\[\sum_{k=1}^\infty 2^{k(1-s)}|a_k|^2
=\sum_{k=1}^\infty 2^{k(1-s)}\frac1{k\,2^{2Ak}}
=\sum_{k=1}^\infty\frac1k=\infty,\]
which implies that $f\notin Q_s$. In this way we have proved
that $\cd^p_\alpha\not\subset Q_s$.

\subsection*{Case 2}

Now, suppose $\frac{1+s}2<\frac{\alpha+1}p$. We can reduce to the previous case
by taking
\[s_0=\frac{2(\alpha+1)}p-1.\]
Indeed, we have that $\frac{1+s_0}2=\frac{\alpha+1}p$, and
\[\frac12\le\frac{1+s}2<\frac{\alpha+1}p<\frac{\alpha+2}p<1,
\qquad1<\frac{2(\alpha+1)}p<2,\]
that is, $0<s_0<1$.
Using case 1, we deduce that $\cd^p_\alpha\not\subset Q_{s_0}$.
Now, observe that
\[\frac{1+s}2<\frac{\alpha+1}p=\frac{1+s_0}2,\]
so we have $s<s_0$ and $Q_s\subset Q_{s_0}$.
It follows that $\cd^p_\alpha\not\subset Q_s$.

\subsection*{Case 3}
Finally, let assume that $\frac{1+s}2>\frac{\alpha+1}p$.
We observe that
\begin{gather*}
p>\alpha+2\iff p-1>\alpha+1\iff1-\frac1p>\frac{\alpha+1}p\iff
\frac1p<1-\frac{\alpha+1}p,\\
\frac{1+s}2>\frac{\alpha+1}p\iff1-\frac{1+s}2<1-\frac{\alpha+1}p\iff
\frac{1-s}2<1-\frac{\alpha+1}p.
\end{gather*}
Then we can take a positive number $q$ satisfying
\begin{equation}\label{ineqmax}
\max\left(\frac1p,\frac{1-s}2\right)<\frac1q<1-\frac{\alpha+1}p.
\end{equation}
Next, we apply Theorem 1.3 of \cite{BKV}, which assures (in
the particular case $t=0$) the following:

\begin{other}\label{inclApalpha}
If \, $-1<\alpha,\beta<\infty$ and \, $0<p,q<\infty$, then $A^p_\alpha\subset A^q_\beta$
\begin{itemize}
\item[(a)]
if $p\le q$ and $\frac{2+\beta}q\ge\frac{2+\alpha}p$,
\end{itemize}
\vspace{-5pt}
or
\vspace{-5pt}
\begin{itemize}
\item[(b)]
if $p>q$ and $\frac{1+\beta}q>\frac{1+\alpha}p$.
\end{itemize}
\end{other}
\noindent
Take $\beta=q-2$. Using \eqref{ineqmax} we have
\[\frac1q<1-\frac{\alpha+1}p<1,\]
which implies that $q>1$ and $\beta>-1$. It also follow from
\eqref{ineqmax} that $p>q$ and
\[\frac{1+\beta}q=\frac{q-1}q=1-\frac1q>\frac{\alpha+1}p.\]
Applying part (b) of Theorem~\ref{inclApalpha}, we deduce that $A^p_\alpha\subset A^q_\beta$. It follows that
\[\cd^p_\alpha\subset D^q_\beta=D^q_{q-2}\,.\]
Now, observe that $D^q_{q-2}$ is the Besov space $B^q$, as $q>1$.
Then we have that $\cd^p_\alpha\subset B^q$.
Finally, we use the inequality $\frac{1-s}2<\frac1q$ given
in \eqref{ineqmax} to prove that $B^q\subset Q_s$.
Indeed, we can use Proposition~\ref{BpcontQs} if $s>0$, and in the case $s=0$ we have that $q<2$, which implies that $B^q\subset B^2=\cd=Q_s$.
Then we have the inclusion $\cd^p_\alpha\subset Q_s$
and the proof of Proposition~\ref{prop3} is complete.

\section{Superposition between $Q_s$ and $\cd^p_\alpha$ spaces, $p-2<\alpha<p-1$}

The remaining cases for $\cd^p_\alpha$ spaces correspond to
the index $\alpha$ between $p-2$ and $p-1$.
For superposition operators from $\cd^p_\alpha$ to
$Q_s$ we have the following.

\begin{theorem}
Let $\varphi$ be an entire function, $0\le s<\infty$, $p>0$ and $\alpha>-1$, such that
$p-2<\alpha<p-1$. Then, $S_\varphi(\cd^p_\alpha)\subset Q_s$ if and only if
$\varphi$ is constant.
\end{theorem}

\begin{proof}
First, observe that it is enough to prove this result for $s>1$, that is, for
the Bloch space $\cb$, because $Q_s\subset\cb$. In this case, the proof can
be done following
part (d) of Theorem~\ref{BMOABlochDp}, using the function
given by
\[f(z)=\frac1{(1-z)^\beta},\qquad z\in\D,\]
where $\beta$ can be chosen satisfing the inequality $0<\beta<\frac{2+\alpha}p-1$
(observe that $\frac{2+\alpha}p-1>0$).
\end{proof}

In the other way, from $Q_s$ to $\cd^p_\alpha$ spaces,
we have characterize superposition operators for $s=1$
(the space $BMOA$), $s>1$ (the Bloch space) and $s=0$
(the Dirichet space). Now, for $0<s<1$ we have solved
the problem if $p<2$. In the case $p\ge2$ we
have some open cases.

For the space $BMOA$, the corresponding result is included in the following
theorem, proved using lacunary series.

\begin{theorem}\label{bmoadpalpha}
Let $\varphi$ be an entire function, $0\le s\le1$, $p>0$ and $\alpha>-1$,
such that $p-2<\alpha<p-1$. If one of the following conditions
\begin{itemize}
\item[(a)]
$\alpha<\frac{p(s+1)}2-1$,
\item[(b)]
$p<2$ and $\alpha=\frac{p(s+1)}2-1$,
\end{itemize}
is satisfied, then $S_\varphi(Q_s)\subset\cd^p_\alpha$ if and
only if $\varphi$ is constant.
\end{theorem}

Observe that, in the case $s=1$, the condition (a) is satisfied ($\alpha<p-1$). Then
we have that  $S_\varphi(BMOA)\subset\cd^p_\alpha$ if and
only if $\varphi$ is constant.
We also note that the equality in (b) implies
that $s<1$.

\begin{proof}
Let us define the function
\[f(z)=\sum_{k=1}^\infty \frac1{k^A\,2^{\frac{k(1-s)}2}}\,z^{2^k},\qquad z\in\D,\]
where $A=2$ if (a) holds, or $A=\frac1p$ if (b) is true.
Then $f$ is given by a lacunary series $\sum_{k=0}^\infty a_kz^{n_k}$, and we see that
\[\sum_{k=0}^\infty|a_kz^{n_k}|\le\sum_{k=0}^\infty|a_k|
=\sum_{k=1}^\infty\frac1{k^A\,2^{\frac{k(1-s)}2}},\]
for all $z\in\D$. Now, if (a) holds, we see that
\[\sum_{k=1}^\infty\frac1{k^A\,2^{\frac{k(1-s)}2}}
\le\sum_{k=1}^\infty\frac1{k^A}
=\sum_{k=1}^\infty\frac1{k^2}<\infty,\]
and if  we have (b), then $s<1$ and
\[\sum_{k=1}^\infty\frac1{k^A\,2^{\frac{k(1-s)}2}}
\le\sum_{k=1}^\infty\frac1{2^{\frac{k(1-s)}2}}
<\infty.\]
Then $f$ is an analytic function in $\D$, and
we have that $f\in H^\infty$.

Let us prove that $f\in Q_s$ and $f\notin\cd^p_\alpha$.
If $s=1$, as $f\in H^\infty$ and $H^\infty\subset BMOA=Q_1$, we have that
$f\in Q_1$. If $0\le s<1$, we can apply \eqref{qslacunar}. Using that $2A>1$, in both (a) and (b), we see that
\[\sum_{k=0}^\infty2^{k(1-s)}\sum_{j:\,n_j\in I_k}|a_j|^2
=\sum_{k=0}^\infty2^{k(1-s)}|a_k|^2
=\sum_{k=0}^\infty2^{k(1-s)}\frac1{k^{2A}\,2^{k(1-s)}}
=\sum_{k=0}^\infty\frac1{k^{2A}}<\infty.\]
Then we have that $f\in Q_s$.
Now we see that
\[\sum_{k=1}^\infty\frac{|a_k|^p}{n_k^{\alpha+1-p}}
=\sum_{k=1}^\infty\frac1{k^{Ap}\,2^{\frac{k(1-s)p}2}(2^k)^{\alpha+1-p}}
=\sum_{k=1}^\infty\frac{\left(2^{\frac{p(s+1)}2-1-\alpha}\right)^k}{k^{Ap}}.\]
If (a) is true, then $2^{\frac{p(s+1)}2-1-\alpha}>1$ and this
sum is infinity. If (b) is true, then the sum is
$\sum_{k=1}^\infty\frac1{k^{Ap}}=\sum_{k=1}^\infty\frac1k=\infty$. Then, using \eqref{Dpalphalac}, we have that $f\notin\cd^p_\alpha$.

Now, if we assume that $\varphi$ is not constant, we
can follow the proof of part (ii) of Proposition~\ref{prop2},
interchanging $Q_s$ and $D^p_\alpha$, and get a
contradiction.
\end{proof}

For the Bloch and Dirichlet spaces
we have the following.

\begin{theorem}
Let $\varphi$ be an entire function, $p>0$ and $\alpha>-1$,
such that $p-2<\alpha<p-1$. Then
\begin{itemize}
\item[(a)]
$S_\varphi({\cb})\subset D^p_\alpha$ if and only if $\varphi$ is
constant.
\item[(b)]
If $p\le2$ and $\alpha>\frac p2-1$, then
$S_\varphi({\cd})\subset\cd^p_\alpha$ if and only if $\varphi$ is of order less than $2$, or of order $2$ and finite
type.
\item[(c)]
If $p>2$, then $S_\varphi({\cd})\subset\cd^p_\alpha$ if and only if $\varphi$ is of order less than $2$, or of order $2$ and finite
type.
\end{itemize}
\end{theorem}

Let us remark that if $p=2$ then $\alpha>p-2=0=\frac p2-1$. Then with this result and
Theorem~\ref{bmoadpalpha} in the case $s=0$ we have all the cases for the Dirichlet
space.

\begin{proof}
Part (a) follows trivially from Theorem \ref{BMOABlochDp} (b),
as $\alpha<p-1$ implies that $\cd^p_\alpha\subset\Dp$.
Now, Theorem 24 of \cite{BFV} for $p=2$ gives (b).

Next we prove (c). If $S_\varphi({\cd})\subset\cd^p_\alpha$ then, using
again that $\cd^p_\alpha\subset\Dp$,
we have that $S_\varphi({\cd})\subset\Dp$,
and, by part (a) of
Theorem \ref{DirichletDp}, we see that $\varphi$ is of order less than $2$, or of order $2$ and finite type. In the other way, let us assume that $\varphi$ is of
this order and type. Observe that
\[p-2<\alpha\iff p<2+\alpha\iff
1<\frac{2+\alpha}p\iff
\frac12<\frac{2+\alpha}p-\frac12.\]
Let us take a positive number $q$ satisfying
\[\frac12<\frac1q<\frac{2+\alpha}p-\frac12,\]
so that
\[1<q\left(\frac{2+\alpha}p-\frac12\right),
\qquad\frac q2-1<q\,\frac{2+\alpha}p-2.\]
Let $\beta=q\,\frac{2+\alpha}p-2$. Then
\[\frac{2+\beta}q=\frac{2+\alpha}p\quad\text{and}\quad
\beta>\frac q2-1>-1.\]
Now, as $p>2$ and $q<2$, we also have that $q<p$.
In these conditions we can apply two results.
We use again Theorem 24 of \cite{BFV} for $p=2$, to see that $S_\varphi(\cd)\subset\cd^q_\beta$, and we can apply part (a) of
Theorem \ref{inclApalpha} to deduce that
$A^q_\beta\subset A^p_\alpha$, which implies that
$\cd^q_\beta\subset\cd^p_\alpha$. Then we have that
$S_\varphi(\cd)\subset\cd^p_\alpha$ and the proof is finished.
\end{proof}

For the spaces $Q_s$, $0<s<1$, we have already solved some cases of our problem  in Theorem~\ref{bmoadpalpha}. Now we can see that
the proof of Theorem~\ref{QsDp} (a) for the
case $p<2$ can also be done for $\alpha\in(p-2,p-1)$
instead of $\alpha=p-1$, if the additional
condition $\alpha>p(s+1)/2-1$ is required. Then we have the following.

\begin{theorem}\label{QsDpalpha}
Let $\varphi$ be an entire function, $0<s<1$, $0<p<2$ and
$\alpha>-1$, such that $p-2<\alpha<p-1$ and
$\alpha>\frac{p(s+1)}2-1$. Then
$S_\varphi(Q_s)\subset\cd^p_\alpha$ if and
only if $\varphi$ is of order less than one, or of order one and type zero.
\end{theorem}

This result, together with Theorem~\ref{bmoadpalpha}, give the answer to our problem for
$p<2$.
Now we turn our attention to the case $p\ge2$.

Theorem~\ref{bmoadpalpha} proves that $S_\varphi(Q_s)\subset\cd^p_\alpha$ only
for constant functions $\varphi$ if $\alpha<p(s+1)/2-1$.
Our next result solves the problem if $\alpha>p-2+s$.
Let us notice that
$p(s+1)/2-1\le p-2+s$, then we have the remaining cases
\[\frac{p(s+1)}2-1\le\alpha\le p-2+s.\]
In particular, if
$p=2$, we have that $p(s+1)/2-1=s=p-2+s$, and the only
remaining case is $\alpha=s$.
The mentioned result is the following.

\begin{theorem}\label{p-2+s}
Let $\varphi$ be an entire function, $0<s<1$, $p\ge2$ and
$\alpha>-1$, such that $p-2+s<\alpha<p-1$. Then
$S_\varphi(Q_s)\subset\cd^p_\alpha$ if and
only if $\varphi$ is of order less than one, or of order one and type zero.
\end{theorem}

\begin{proof}
Suppose that $S_\varphi(Q_s)\subset\cd^p_\alpha$. Using that $\alpha<p-1$ we
have that $\cd^p_\alpha\subset\Dp$. Then $S_\varphi(Q_s)\subset\Dp$ and we
can apply (a) of Theorem~\ref{QsDp}.

Now, let us assume that $\varphi$ is of order less than one, or of order one and type zero. We take $f\in Q_s$ and we have to prove that $S_\varphi(f)\in\cd^p_\alpha$.

As $f\in Q_s\subset\cb$, there exists a positive constant $M$ such that
\begin{equation}\label{ffprimebloch}
|f(z)|\le M\left(\log\frac1{1-|z|}+1\right),
\qquad
|f'(z)|\le\frac M{1-|z|^2},
\qquad\text{for all $z\in\D$.}
\end{equation}
Let $\varepsilon$ be the positive number
\[\varepsilon=\frac{\alpha-s-p+2}{Mp}.\]
We know that $\varphi'$
is also of order less than one, or of order one and type zero, then there
exists $A>0$ such that
\[
|\varphi'(z)|\le A\,e^{\varepsilon|z|},\qquad
\text{for all $z\in\C$.}
\]
So, for all $z\in\D$ we have that
\[|\varphi'(f(z))|\le A\,e^{\varepsilon|f(z)|}
\le A\,e^{\varepsilon M\left(\log\frac1{1-|z|}+1\right)}
=A\,e^{\varepsilon M}\frac1{(1-|z|)^{\varepsilon M}}.\]
Using this inequality and the second one of \eqref{ffprimebloch}, we deduce that
\begin{align*}
&\int_\D(1-|z|^2)^\alpha\bigl|\bigl(S_\varphi(f)\bigr)'(z)\bigr|^p\,dA(z)
=\int_\D(1-|z|^2)^\alpha|f'(z)|^p|\varphi'(f(z))|^p\,dA(z)\\[5pt]
&\le A^pe^{\varepsilon Mp}\int_\D(1-|z|^2)^\alpha|f'(z)|^p\frac1{(1-|z|)^{\varepsilon Mp}}\,dA(z)\\[5pt]
&=A^pe^{\varepsilon Mp}\int_\D(1-|z|^2)^s|f'(z)|^2
(1-|z|^2)^{\alpha-s}|f'(z)|^{p-2}
\frac1{(1-|z|)^{\varepsilon Mp}}\,dA(z)\\[5pt]
&\le A^pe^{\varepsilon Mp}\int_\D(1-|z|^2)^s|f'(z)|^2
(1-|z|^2)^{\alpha-s}\left(\frac M{1-|z|^2}\right)^{p-2}
\frac1{(1-|z|)^{\varepsilon Mp}}\,dA(z)\\[5pt]
&\le A^pe^{\varepsilon Mp}M^{p-2}c_{p,\alpha,s}\int_\D(1-|z|)^s|f'(z)|^2
(1-|z|)^{\alpha-s-\varepsilon Mp-p+2}\,dA(z),
\end{align*}
where $c_{p,\alpha,s}$ is a positive constant which only
depends on $p$, $\alpha$ and $s$. By our choice of
$\varepsilon$, we have that the exponent of $(1-|z|)$
in the last expression is zero, so the integral becomes
\[\int_\D(1-|z|)^s|f'(z)|^2\,dA(z),\]
which is finite because $f\in Q_s$ (see
\eqref{sCarleson}). Then we have proved that
\[\int_\D(1-|z|^2)^\alpha\bigl|\bigl(S_\varphi(f)\bigr)'(z)\bigr|^p\,dA(z)<\infty,\]
that is, $S_\varphi(f)\in\cd^p_\alpha$. This completes the proof.
\end{proof}
\par\bigskip
{\bf Acknowledgements.} The authors wish to thank the referee for
his/her comments and suggestions for improvement.

\end{document}